\documentclass[11pt]{article}
\usepackage{amssymb,amsmath,amsfonts,amsthm,color}

\usepackage{epsfig}

\numberwithin{equation}{section}

\newcommand{\R}{{\mathbb R}}

\newtheorem{theorem}{Theorem}[section]

\newtheorem{lemma}[theorem]{Lemma}

\newtheorem{proposition}[theorem]{Proposition}
\newtheorem{remark}[theorem]{Remark}

\begin{document}

\title{\vskip-0.3in Quasilinear elliptic inequalities with Hardy potential and nonlocal terms}

\author{Marius Ghergu\footnote{School of Mathematics and Statistics,
    University College Dublin, Belfield, Dublin 4, Ireland; {\tt
      marius.ghergu@ucd.ie}}\;\,\footnote{Institute of Mathematics Simion Stoilow of the Romanian Academy, 21 Calea Grivitei St., 010702 Bucharest, Romania}
      $\;\;$        $\;$
{Paschalis Karageorgis\footnote{School of Mathematics,
    Trinity College Dublin; {\tt
        pete@maths.tcd.ie}}}
 $\;\;$    and    $\;$
{Gurpreet Singh\footnote{School of Mathematics,
    Trinity College Dublin; {\tt
        gurpreet.bajwa2506@gmail.com}}}
}

\maketitle

\begin{abstract}
We study the quasilinear elliptic inequality 
$$
-\Delta_m u - \frac{\mu}{|x|^m}u^{m-1} \geq (I_\alpha*u^p)u^q \quad\mbox{ in }\mathbb{R}^N\setminus \overline B_1, N\geq 1,
$$
where $p>0$, $q, \mu \in \mathbb{R}$, $m>1$ and $I_\alpha$ is the Riesz potential of order $\alpha\in (0,N)$. We obtain necessary and sufficient conditions for the existence of positive solutions. 
\end{abstract}


\medskip

\section{Introduction and the main results}\label{sec1}
In this paper we are concerned with the following quasilinear elliptic inequality
\begin{equation}\label{go}
-\Delta_m u - \frac{\mu}{|x|^m}u^{m-1} \geq (I_\alpha*u^p)u^q \quad\mbox{ in }\R^N\setminus \overline B_1, N\geq 1,
\end{equation}
where $B_1\subset \R^N$ denotes the open unit ball, $\Delta_m u= {\rm div}(|\nabla u|^{m-2}\nabla u)$, $m>1$ is the $m-$Laplace operator  of $u$ and $p>0$, $q,\mu \in \R$, $\alpha\in (0,N)$.

Throughout this paper, $I_\alpha:\R^N\to \R$ denotes the {\it Riesz potential} of order $\alpha\in (0,N)$, $N\geq 1$,  given by (see, e.g., \cite{LL2001})
\begin{equation}\label{rieszdef}
I_\alpha(x)=\frac{A_\alpha}{|x|^{N-\alpha}}\,,\quad \mbox{ with }\; A_\alpha=\frac{\Gamma\big(\frac{N-\alpha}{2}\big)}{\Gamma(\frac{\alpha}{2}\big) \pi^{N/2}2^\alpha}= C(N, \alpha)> 0,
\end{equation}
and $I_{\alpha}*u^{p}$ is the convolution operation which is given by
$$
(I_{\alpha}*u^{p})(x)= \int_{\R^N\setminus \overline{B}_1}I_{\alpha}(x-y)u^{p}(y) dy.
$$

We say that $u\in W_{loc}^{1, m}(\R^N\setminus \overline B_1) \cap C(\R^N\setminus \overline B_1) $ is a positive  solution of \eqref{go} if
\begin{enumerate}
\item[(i)]  $(I_\alpha*u^p)u^q\in L^{1}_{loc}(\R^N\setminus \overline B_1)$ and $u>0$;
\item[(ii)] $u$ satisfies
\begin{equation}\label{c1}
\int_{\R^N\setminus \overline B_1}\frac{u^{p}(y)}{1+|y|^{N-\alpha}}dy< \infty;
\end{equation}

\item[(iii)] for any $\phi \in C_{c}^{\infty}(\R^N\setminus \overline B_1)$, $\phi\geq 0$ we have
\begin{equation*}
\int_{\R^N\setminus \overline B_1} \Big( |\nabla u|^{m-2}\nabla u\nabla \phi- \frac{\mu}{|x|^{m}}u^{m-1}\phi\Big)
 \geq \int_{\R^N\setminus \overline B_1}(I_\alpha*u^p)u^q \phi.
\end{equation*}
\end{enumerate}
Note that condition \eqref{c1} is needed since we require $I_\alpha*u^p< \infty$ a.e. in $\R^N\setminus \overline B_1$.

An important role in our analysis will be played by the quantity
\begin{equation}\label{chh}
C_{H}=\Big|\frac{N-m}{m}\Big|^m,
\end{equation}
which is the optimal constant in the Hardy inequality 
\begin{equation}\label{hardy1}
\int_{\R^N} |\nabla \phi|^m \geq C_{H}\int_{\R^N} \frac{|\phi|^{m}}{|x|^{m}}\quad\mbox{ for all } \phi\in C^1_c(\R^N).
\end{equation}

We consider the inequality \eqref{go} in $\R^N\setminus \overline B_1$, but our approach can be carried over to the case where \eqref{go} is posed in any  open set $\Omega\subset \R^N$ with the properties: $\R^N\setminus \Omega$ is bounded and $0\not\in \Omega$. In \cite{GKS2020} we studied nonlocal inequalities for general quasilinear operators of type
\begin{equation}\label{paperjde}
-{\rm div}[\mathcal{A}(x, u, \nabla u)] \geq (I_\alpha \ast u^p)u^q
\quad\mbox{ in }\Omega \subset \R^N,
\end{equation}
for three types of open sets $\Omega$ as follows: (i) $\Omega$ is bounded; (ii) $\Omega$ is the exterior of a closed ball; (iii) $\Omega=\R^N$. Also,  the differential operator $\mathcal{A}$ in \eqref{paperjde} is required to satisfy minimal structural assumptions, namely, to be  a {\it weakly-m-coercive} operator, that is,
\begin{equation}\label{wmcc}
\mathcal{A}(x, u, \eta)\cdot\eta \geq c |\mathcal{A}(x, u, \eta)|^{m'}
\quad\mbox{ for all } (x, u, \eta)\in \Omega \times[0, \infty)\times \R^N,
\end{equation} 
where $m'=m/(m-1)>1$ and $c>0$ is a constant. 

The present work, already announced in \cite{GKS2020},  investigates the influence of the Hardy term $\mu u^{m-1}/|x|^m$ in \eqref{go} as well as the more general inequality
\begin{equation}\label{g1}
-{\rm div}[\mathcal{A}(x, u, \nabla u)] - \frac{\mu}{ |x|^{\theta}} u^{m-1} \geq (I_\alpha \ast u^p)u^q
\quad\mbox{ in } \R^N\setminus \overline{B}_1,
\end{equation}
where $\mu,\theta\in\R$ and $\mathcal{A}$ is {\it weakly-m-coercive} operator as given in \eqref{wmcc}.

Our study is motivated by the semilinear elliptic inequality
\begin{equation}\label{ch}
-\Delta u + \frac{\lambda}{|x|^{\gamma}}u \geq (I_\alpha*u^p)u^q \quad\mbox{ in } \R^{N}\setminus B_1,
\end{equation}
which was considered in \cite{MV2013}. The equation
$$
-\Delta u + \lambda u =(I_\alpha*u^p)u^q
$$
is known in the literature as the Choquard or Choquard-Pekar equation and arises in many areas of mathematical modelling of real life phenomena with direct links to quantum physics.
We only mention here the recent works \cite{CDM2008, CZ2016, CZ2018, FG2020, GS2019, GT2016, MV2013, MV2017,S2019, WH2019} to illustrate this aspect.

Another motivation for our study comes from \cite{LLM2007} where the authors consider the inequality
\begin{equation*}
-\Delta_p u - \frac{\mu}{|x|^p}u^{p-1} \geq \frac{C}{|x|^{\sigma}}u^q,
\end{equation*}
in exterior domains of $\R^N$, $N\geq 2$ where $p>1$,  $q, \sigma, \mu\in \R$. We also refer the reader to \cite{DM2010, MP2001} for a wide range of quasilinear elliptic inequalities.

Let us point out that our approach is distinct to that in \cite{MV2013}. One key argument in the approach of \eqref{ch}  is the nonlocal version of the Agmon-Allegretto-Piepenbrink positivity principle (see \cite[Proposition 3.2]{MV2013}). As we have already emphasized in \cite{GKS2020}, this approach relies essentially on the linear character of the differential operator in \eqref{ch} and does not carry over to a quasilinear setting such as \eqref{go}. In turn, we provide useful a priori estimates (see Lemma \ref{l101} in Section 3) which  allow us to obtain optimal conditions for the existence of a solution to \eqref{go}. 

The existence of a solution to \eqref{go} is related to the power-type equation
\begin{equation}\label{bet}
- \beta |\beta|^{m-2} (\beta(m-1) + N-m)=\mu,
\end{equation}
which has real solutions if and only if $\mu\leq C_H$. In such a case, we denote
by $\beta^-\leq \beta^+$ the two solutions of equation \eqref{bet}, which are distinct if $\mu<C_H$.

If $N>m$ we have the following result concerning \eqref{go}.

\begin{theorem}\label{thmain1}
Assume $N>m>1$, $p>0$, $q,\mu\in \R$ and $\alpha\in (0, N)$. Then \eqref{go} has a positive solution if and only if the following conditions hold:
\begin{enumerate}
\item[{\rm(i)}] $\mu\leq C_H$;
\item[\rm (ii)] $p> \frac{\alpha}{|\beta^{-}|} $;
\item[\rm (iii)] $p+q>m-1+\frac{m+\alpha}{|\beta^-|}$;
\item[{\rm (iv)}] $
\left\{
\begin{aligned}
& q> m-1-\frac{N-m-\alpha}{|\beta^{-}|} &&\quad \mbox{ if } \alpha>N-m;\\
&q\geq m-1 &&\quad \mbox{ if } \alpha=N-m;\\
& q> m-1 - \frac{N-m-\alpha}{N}p &&\quad \mbox{ if } \alpha<N-m;\\
&q> m-1-\frac{N-m-\alpha}{|\beta^+|} &&\quad \mbox{ if } \alpha<N-m \mbox{ and } \mu>0.
\end{aligned}
\right.
$
\end{enumerate}
\end{theorem}

\begin{figure}[h!]
\begin{center}
  \includegraphics[width=6.0in]{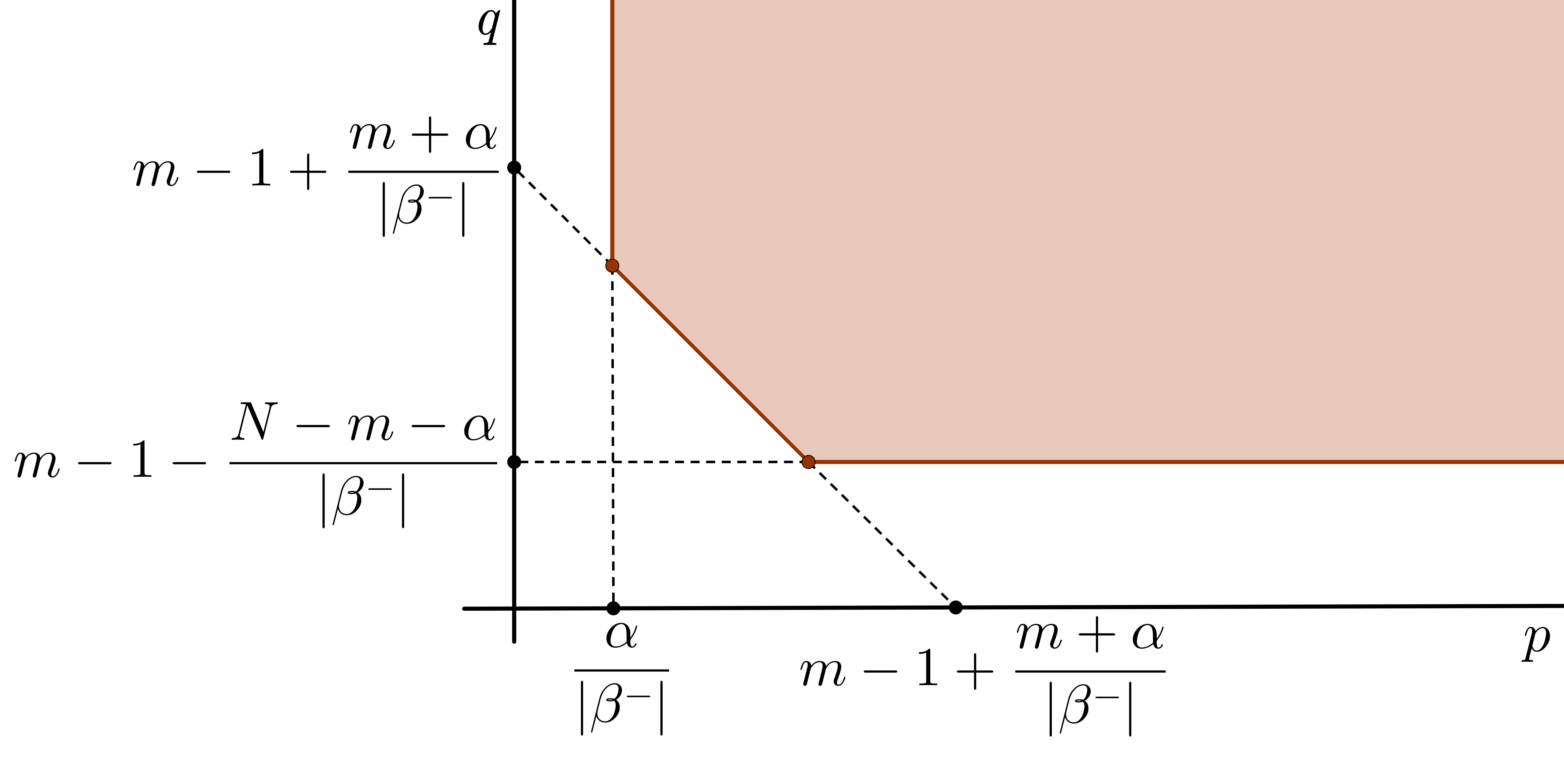}
  \caption{The existence region (shaded) for positive solutions to \eqref{go} in the case $N>m>1$,  $\alpha\geq N-m$} and $\mu\leq C_H$.
  \label{fig:case1}
  \end{center}
\end{figure}

\vspace{-0.3in}

Figure 1 above concerns the case  $\alpha\geq N-m>0$ and $\mu\leq C_H$ in Theorem \ref{thmain1}. The shaded region depicts the existence set in the $pq$ plane. Figures 2--3 below illustrate the existence region in Theorem \ref{thmain1} for $0<\alpha<N-m$.
These are in line with \cite[Theorem 9]{MV2013} related to the semilinear inequality \eqref{ch}.

\begin{figure}[h!]
\begin{center}
  \includegraphics[width=5.5in]{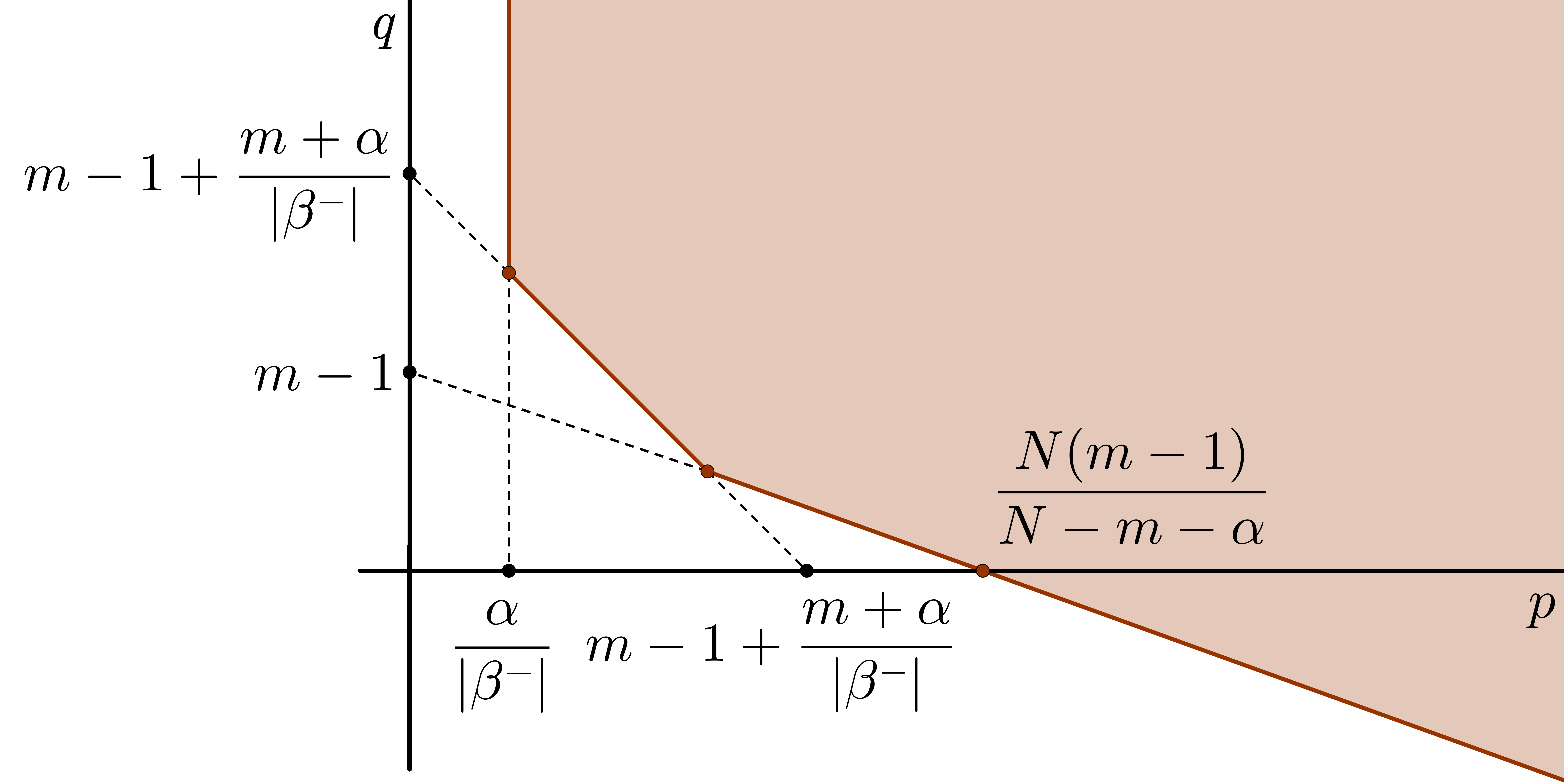}
  \caption{The existence region (shaded) for positive solutions to \eqref{go} in the case $N>m>1$,  $\alpha< N-m$} and $\mu\leq 0$.
  \label{fig:case2_1}
  \end{center}
\end{figure}

\begin{figure}[h!]
\begin{center}
 \includegraphics[width=5.5in ]{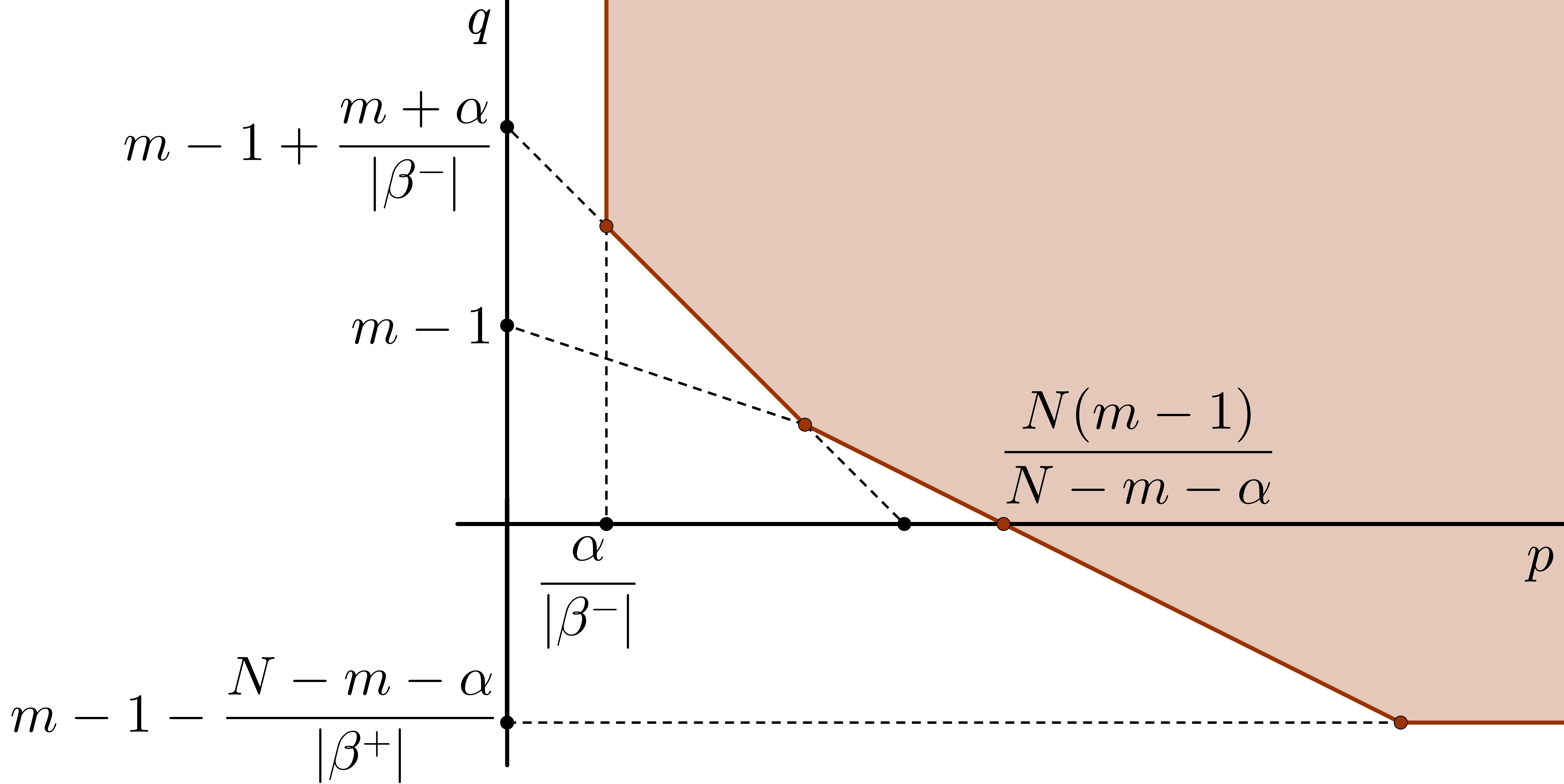}
  \caption{The existence region (shaded) for positive solutions to \eqref{go} in the case $N>m>1$,  $\alpha<N-m$} and $0<\mu\leq C_H$.
  \label{fig:case2_2}
  \end{center}
\end{figure}

\medskip

If $N\leq m$ we have the following result concerning \eqref{go}.

\begin{theorem}\label{thmain2}
Assume  $m>1$, $m\geq N\geq 1$,  $p>0$, $q,\mu\in \R$ and $\alpha\in (0, N)$. Then \eqref{go} has a positive solution if and only if the following conditions hold:

\begin{enumerate}
\item[{\rm(i)}] $\mu<0$;
\item[\rm (ii)] $p> \frac{\alpha}{|\beta^{-}|} $;
\item[\rm (iii)] $p+q>m-1+\frac{m+\alpha}{|\beta^-|}$;
\item[{\rm (iv)}] $q> m-1-\frac{N-m-\alpha}{|\beta^{-}|}$.
\end{enumerate}
\end{theorem}

The existence set in Theorem \ref{thmain2} is similar to the one depicted in Figure 1 above.
 
The remaining of our paper is organised as follows. In Section 2 we collect some preliminary facts which will be used in our approach. In Section $3$ we study the nonexistence results of solutions to the general quasilinear elliptic inequality \eqref{g1}. Finally, Section 4 contains the proofs of our main results.

\section{Preliminary results}

In this section we collect some preliminary facts that we use in our approach. We start with a
basic result on the behaviour of the function
\begin{equation}\label{G}
G(\beta) = - \beta |\beta|^{m-2} (\beta(m-1) + N-m)\,,\quad\beta\in \R.
\end{equation}

\begin{lemma} \label{mu}
Let $N\in \R$, $m>1$ and $C_H = \left| \frac{N-m}{m} \right|^m$.
\begin{enumerate}
\item[{\rm (i)}]
One has $G(\beta) \leq C_H$ for all $\beta\in\R$ with equality if and only if $\beta= \beta_* :=
\frac{m-N}{m}$.

\item[{\rm (ii)}]
If $\mu\leq C_H$, then the set of all real numbers $\beta$ such that $G(\beta) \geq \mu$ is a
finite interval $[\beta^- , \beta^+]$ with $\beta^- \leq \beta_* \leq \beta^+$.  In addition, one
has $\beta^- = \beta_* = \beta^+$ when $\mu=C_H$.

\item[{\rm (iii)}]
If $N>m$ and $\mu \leq C_H$, then $\beta^- \leq \beta^* <0$.  If $N\leq m$ and $\mu<0$, then
$\beta^- < 0 < \beta^+$.  If $N\leq m$ and $0\leq \mu\leq C_H$, then $0\leq \beta^- \leq
\beta^*$.

\end{enumerate}
\end{lemma}

\begin{proof}
Using the definition \eqref{G}, one easily find that
\begin{equation*}
G'(\beta) = -(m-1) |\beta|^{m-2} (m\beta + N -m).
\end{equation*}
In particular, $G$ is increasing on $(-\infty, \beta_*)$ and decreasing on $(\beta_*, \infty)$, so
its maximum value is $G(\beta_*) = C_H$.  Since $G(0)= 0$ and $G(\beta)\to -\infty$ as $\beta\to
\pm\infty$, the result follows easily.
\end{proof}

\begin{proposition}\label{p1}{\rm (See \cite[Theorem 3.4]{LLM2007})}
The inequality
\begin{equation}\label{eqA0}
-\Delta_m u - \frac{\mu}{|x|^m}u^{m-1} \geq 0 \quad\mbox{ in }\R^N\setminus \overline B_1,
\end{equation}
has positive solutions if and only if $\mu \leq C_H$, where $C_H$ is the optimal Hardy constant defined in \eqref{chh}. Furthermore, if $\mu \leq C_H$ and $u$ is a positive solution to \eqref{eqA0}, then there exists a constant $c>0$ such that
\begin{equation}\label{betam}
u\geq c|x|^{\beta^{-}} \quad\mbox{ in } \R^N\setminus B_2.
\end{equation}

\end{proposition}
Our next result concerns the nonexistence of positive solutions to
\begin{equation}\label{eqA}
-\Delta_m u - \frac{\mu}{|x|^m}u^{m-1} \geq  \frac{C}{|x|^{\sigma}}u^{q} \quad\mbox{ in } \R^N\setminus \overline B_1.
\end{equation}

\begin{proposition}\label{p2}{\rm (See \cite[Sections 4.1 and 4.3]{LLM2007})}

Let $C>0$, $m>1$ and $\mu\leq C_H$. If one of the following conditions hold:

\begin{enumerate}
\item[{\rm (i)}] $q\geq m-1$ and $\sigma \leq \beta^{-}(q-m+1)+m$;

\item[{\rm (ii)}] $q< m-1$,  $\mu<C_H$ and $\sigma \leq \beta^{+}(q-m+1)+m$;

\item[{\rm (iii)}] $q< m-1$,  $\mu=C_H$ and $\sigma < \beta^{+}(q-m+1)+m$;

\item[{\rm (iv)}] $-1\leq q<m-1$, $\mu=C_H$ and  $\sigma =\beta^{+}(q-m+1)+m$;

\end{enumerate}
then \eqref{eqA} has no positive solution.
\end{proposition}

\begin{lemma}\label{lbas}
Let $\alpha\in (0,N)$, $p>0$ and $f\in L^1_{loc}(\R^N\setminus \overline{B}_1)\cap C(\R^N\setminus \overline{B}_1)$, $f\geq 0$.

\begin{enumerate}
\item[\rm (i)] There exists $C>0$ such that $$ (I_\alpha*f)(x)\geq C|x|^{\alpha-N}\quad\mbox{ for any }x\in \R^N\setminus B_2. $$

\item[\rm (ii)]
If $f(x)\geq c|x|^{\beta}$ in $\R^N\setminus B_1$ for some $c> 0$, then,
$$ \left\{
\begin{aligned}
&  (I_\alpha*f^p)(x)=\infty && \quad\mbox{ if } \; \alpha+p\beta\geq 0 \\
&  (I_\alpha*f^p)(x)\geq C|x|^{\alpha+p\beta} && \quad\mbox{ if } \; \alpha+p\beta< 0
\end{aligned}
\right. \quad\mbox{ in }\; \R^N\setminus B_1,
$$
for some $C>0$.
\item[\rm (iii)]
If $f(x)\leq c|x|^{\gamma}\log^\tau(s|x|)$ in $\R^N\setminus B_1$ for some $c>0$, $\tau\geq 0>\gamma$, $s>1$ and $\alpha+p\gamma<0$,  then
$$
(I_\alpha*f^p)(x)\leq
C\left\{
\begin{aligned}
&|x|^{\alpha+p\gamma}\log^{p\tau}(s|x|)&&\quad\mbox{ if }p|\gamma|<N\\
&|x|^{\alpha-N}\log^{1+p\tau}(s|x|)&&\quad\mbox{ if }p|\gamma|=N\\
&|x|^{\alpha-N}&&\quad\mbox{ if }p|\gamma|>N
\end{aligned}
\right.
 \quad\mbox{ in }\; \R^N\setminus B_1.
$$

\end{enumerate}

\end{lemma}
\begin{proof}
(i) For any $x\in \R^N\setminus B_2$ we have
$$
(I_\alpha*f)(x)\geq C\int\limits_{3/2< |y|<2}\frac{f(y)}{|x-y|^{N-\alpha}}dy \geq C\int\limits_{3/2< |y|< 2}\frac{f(y)}{|2x|^{N-\alpha}}dy =
\frac{C}{|x|^{N-\alpha}}.
$$

(ii) We have $$ (I_\alpha*f^p)(x)\geq C\int\limits_{|y|\geq2|x|}\frac{|y|^{p\beta}}{|x-y|^{N-\alpha}}dy \geq C\int\limits_{|y|\geq 2|x|} |y|^{\alpha-N+p\beta} dy= C\int\limits_{2|x|}^{\infty}t^{\alpha+p \beta}\frac{dt}{t}, $$

and the conclusion follows.

(iii) For $x\in \R^N\setminus B_1$  we estimate
\begin{equation*}
(I_\alpha*f^p)(x) \leq \int_{|y|\geq 2|x|} \frac{A_\alpha f(y)^p \,dy}{|x-y|^{N-\alpha}} +
\int_{\frac12 |x| \leq |y| \leq 2|x|} \frac{A_\alpha f(y)^p \,dy}{|x-y|^{N-\alpha}} +
\int_{\frac12 |x| \geq |y|} \frac{A_\alpha f(y)^p \,dy}{|x-y|^{N-\alpha}}.
\end{equation*}
For the first integral on the right-hand side, one has $|y|\geq 2|x|$, so $|x-y| \geq |y|-|x| \geq \frac12 |y|$ and
\begin{equation*}
\int_{|y|\geq 2|x|} \frac{A_\alpha f(y)^p \,dy}{|x-y|^{N-\alpha}}
\leq C \int_{|y|\geq 2|x|} \frac{\log^{p\tau} (s|y|) \,dy}{|y|^{N-\alpha-p\gamma}}
\leq C|x|^{\alpha + p\gamma} \log^{p\tau}(s|x|).
\end{equation*}
The exact same estimate holds for the second integral because
\begin{align*}
\int_{\frac12 |x| \leq |y| \leq 2|x|} \frac{A_\alpha f(y)^p \,dy}{|x-y|^{N-\alpha}}
&\leq C|x|^{p\gamma} \log^{p\tau}(s|x|) \int_{\frac12 |x| \leq |y| \leq 2|x|}
\frac{dy}{|x-y|^{N-\alpha}} \\
&\leq C|x|^{\alpha + p\gamma} \log^{p\tau}(s|x|).
\end{align*}
Next, we turn to the third integral.  If $\frac12 |x|\geq |y|$, then
$|x-y| \geq |x|-|y| \geq \frac12 |x|$, so
\begin{equation*}
\int_{\frac12 |x| \geq |y|} \frac{A_\alpha f(y)^p \,dy}{|x-y|^{N-\alpha}} \leq
C|x|^{\alpha-N} \log^{p\tau} (s|x|) \int_{\frac12 |x| \geq |y|} |y|^{p\gamma} \,dy.
\end{equation*}
The result now follows by considering the cases $p\gamma>-N$, $p\gamma=-N$ and $p\gamma<-N$.
\end{proof}

\begin{remark}\label{remk}
A direct and useful calculation shows that:
\begin{enumerate}
\item[{\rm (i)}] If $u(x)=\kappa |x|^{\gamma}$, $\gamma\in \R$, $\kappa>0$
then
\begin{equation}\label{radial1}
-\Delta_m u -\frac{\mu}{|x|^m}u^{m-1}= \kappa^{m-1} (G(\gamma)-\mu)|x|^{\gamma(m-1)-m} \quad\mbox{ in }\R^N\setminus B_1,
\end{equation}
where $G$ is defined in \eqref{G}.
\item[{\rm (ii)}] If $u(x)=\kappa |x|^{\gamma} \log^\tau \big(s|x|\big)$ where
$$
\tau\in \R\,,\quad \kappa>0>\gamma\quad \mbox{ and }\quad |\gamma|\log s>\tau,
$$
then
$$
\begin{aligned}
-\Delta_m u -\frac{\mu}{|x|^m}u^{m-1}=&
\kappa^{m-1}|x|^{\gamma(m-1)-m} \log^{\tau(m-1)}\big(s |x|\big)\times\\
&\times \left[-\mu+ \Big( |\gamma|-\frac{\tau} {\log\big(s|x|\big)}\Big)^{m-2}\Big(a+\frac{b}{\log\big(s|x|\big)}+\frac{c}{\log^2\big(s|x|\big)}\Big)\right],
\end{aligned}
$$
where
\begin{equation}\label{abc1}
\begin{aligned}
a&=|\gamma|\big[\gamma(m-1)+(N-m)  \big],\\
b&=-\tau\big[2\gamma(m-1)+(N-m)\big],\\
c&=-\tau(\tau-1)(m-1).
\end{aligned}
\end{equation}
Furthermore,
\begin{equation}\label{radial2}
\begin{aligned}
-\Delta_m u -\frac{\mu}{|x|^m}u^{m-1}=&
\kappa^{m-1}|x|^{\gamma(m-1)-m} \log^{\tau(m-1)}\big(s |x|\big)\times\\
&\times \left[A +\frac{B}{\log\big(s|x|\big)}+\frac{C}{\log^2\big(s|x|\big)}+\dots \right]
\quad\mbox{ in }\R^N\setminus B_1,
\end{aligned}
\end{equation}
where
\begin{equation}\label{abc2}
\begin{aligned}
A&=G(\gamma)-\mu,\\
B&=|\gamma|^{m-3}\big[b|\gamma|-(m-2)a\tau\big]=-\tau|\gamma|(m-1)\big[ m\gamma+(N-m)\big],\\
C&=|\gamma|^{m-4}\Big[ c\gamma^2+b\gamma\tau (m-2)+\frac{(m-2)(m-3)}{2}a\tau^2\Big].
\end{aligned}
\end{equation}
\end{enumerate}

\end{remark}

\section{Non-existence results for general differential operators}

We establish some nonexistence results for solutions to a general quasilinear elliptic inequality
\begin{equation} \label{gog}
-{\rm div}[\mathcal{A}(x, u, \nabla u)] - \frac{\mu}{ |x|^{\theta}} u^{m-1} \geq (I_\alpha \ast u^p)u^q
\quad\mbox{ in } \R^N\setminus \overline B_1,
\end{equation}
where $\mu,\theta\in\R$ and $\mathcal{A}$ is {\it weakly-m-coercive} (in short $W$-$m$-$C$) as stated in \eqref{wmcc}. 
Typical examples that satisfy \eqref{wmcc} are the standard $m$-Laplace operator
$$
\mathcal{A}(x, u, \eta)= |\eta|^{m-2}\eta\,, \quad m>1
$$
and the $m$-mean curvature operator given by
$$
\mathcal{A}(x, u, \eta)= \frac{|\eta|^{m-2}}{\sqrt{1+|\eta|^m}}\eta\,, \quad m\geq 2.
$$
In order to establish the nonexistence results for \eqref{gog} we first obtain
a priori estimates for solutions $u\in C(\R^N\setminus \overline B_1)\cap W^{1,1}_{loc}(\R^N\setminus \overline B_1)$ of the
general inequality
\begin{equation}\label{ff}
-{\rm div}[\mathcal{A}(x, u, \nabla u)]  -\frac{\mu}{ |x|^{\theta}} u^{m-1} \geq f(x)\quad\mbox{ in } \R^N\setminus \overline B_1,
\end{equation}
where $f\in L^1_{loc}(\R^N\setminus \overline B_1)$ and $\mu,\theta\in\R$. Solutions of \eqref{ff} are understood in
the weak sense, that is,
\begin{itemize}
\item $\mathcal{A}(x, u, \nabla u)\in L^{1}_{loc}(\R^N\setminus \overline B_1)^{N}$;
\item  $\frac{\mu}{ |x|^{\theta}} u^{m-1}
\in L^{1}_{loc}(\R^N\setminus \overline B_1)$;
\item for any function $\varphi \in C_{c}^{\infty}(\R^N\setminus \overline B_1)$ with $\varphi\geq 0$, one has
\begin{equation}\label{var}
\int_{\R^N\setminus \overline B_1} \mathcal A(x, u, \nabla u)\cdot \nabla \varphi
- \int_{\R^N\setminus \overline B_1}  \mu  |x|^{-\theta} u^{m-1} \varphi \geq \int_{\R^N\setminus \overline B_1} f(x) \varphi.
\end{equation}
\end{itemize}
In our first result of this section we provide a priori estimates for solutions of \eqref{gog}.

\begin{lemma}\label{l101}
Assume $\mathcal{A}$ is $W$-$m$-$C$ for some $m>1$ and $u$ is a positive solution of \eqref{ff}.
Fix a test function $\phi\in C^\infty_c( \R^N\setminus \overline B_1)$ such that $0\leq \phi\leq 1$ and
\begin{equation} \label{test}
\text{ supp$\,\phi\subset B_{4R}\setminus B_{R/2}$, \quad
$\phi=1$ in $B_{2R}\setminus B_{R}$, \quad
$|\nabla \phi|\leq C/R$ in $\R^N\setminus \overline B_1$, }
\end{equation}
where $C>0$ and $R>2$.
Then, for any $\lambda>m$, there exists a constant $C>0$ independent of $R$ such that
\begin{equation}\label{mainest1}
\int_{\R^N\setminus \overline B_1} f(x) u^{1-m} \phi^\lambda \leq C R^{N-m} + CR^{N-\theta}.
\end{equation}
\end{lemma}

\begin{proof} We follow the method of test functions devised in \cite{BP2001}.
Using $\varphi = u^{1-m} \phi^\lambda$ in \eqref{var}, one finds that
\begin{align} \label{proof0}
\int_{\R^N\setminus \overline B_1} f(x) u^{1-m} \phi^\lambda
&\leq \int_{\R^N\setminus \overline B_1} (1-m)u^{-m} \phi^\lambda \mathcal A(x, u, \nabla u) \cdot \nabla u \notag\\
&\qquad + \int_{\R^N\setminus \overline B_1} \lambda \phi^{\lambda-1} u^{1-m} \mathcal A(x, u, \nabla u) \cdot \nabla \phi - \mu \int_{\R^N\setminus \overline B_1}  |x|^{-\theta} \phi^\lambda.
\end{align}
Since $m>1$, the property \eqref{wmcc} implies
\begin{equation*}
(1-m) u^{-m} \phi^\lambda \mathcal A(x, u, \nabla u) \cdot \nabla u
\leq (1-m)C_1 u^{-m} \phi^\lambda |\mathcal A(x, u, \nabla u)|^{m'}.
\end{equation*}
On the other hand, Young's inequality ensures that
\begin{align*}
\lambda \phi^{\lambda-1} u^{1-m} \mathcal A(x, u, \nabla u)\cdot \nabla \phi
&\leq (m-1)C_1 u^{-m} \phi^\lambda |\mathcal A(x, u, \nabla u)|^{m'} \\
&\qquad +C(\lambda,m) \phi^{\lambda-m}|\nabla \phi|^m.
\end{align*}
Adding the last two estimates and returning to \eqref{proof0}, we conclude that
\begin{equation*}
\int_{\R^N\setminus \overline B_1} f(x) u^{1-m} \phi^\lambda \leq C(\lambda,m)
\int_{\R^N\setminus \overline B_1} \phi^{\lambda-m}|\nabla \phi|^m -\mu \int_{\R^N\setminus \overline B_1} |x|^{-\theta} \phi^\lambda.
\end{equation*}
The result now follows easily using the properties \eqref{test} of the test function $\phi$.
\end{proof}

As a consequence of the above result we find:
\begin{lemma} \label{l1}
Assume $\mathcal A$ is $W$-$m$-$C$ for some $m> 1$ and let $u$ be a positive solution of
\eqref{gog}.  Then
\begin{equation*}
\Big(\int_{B_{2R}\setminus B_1}u^p\Big)\Big(\int_{B_{2R}\setminus B_R}u^{q-m+1}\Big)
\leq CR^{2N-\min(\theta,m)-\alpha} \quad\mbox{ for all }R> 2.
\end{equation*}
\end{lemma}

\begin{proof}
According to Lemma \ref{l101}, one has
\begin{equation*}
\int_{B_{2R}\setminus B_R} (I_\alpha*u^p) u^{q-m+1} \leq CR^{N-\min(\theta,m)}
\quad\mbox{ for all }R> 2.
\end{equation*}
If $x\in B_{2R}\setminus B_R$ and $y\in B_{2R}\setminus B_1$, then $|x-y|\leq |x|+|y|\leq 4R$ so
\begin{equation*}
(I_\alpha*u^p)(x) \geq C\int_{B_{2R}\setminus B_{1}} \frac{u(y)^p}{|x-y|^{N-\alpha}} \,dy
\geq CR^{\alpha-N} \int_{B_{2R}\setminus B_1} u(y)^p \,dy.
\end{equation*}
The result now follows easily by combining the last two estimates.
\end{proof}

We are now ready to state and prove the main result of this section which concerns the nonexistence of positive solutions to inequality \eqref{gog}.

\begin{theorem}\label{thm2}
Assume $\mathcal{A}$ is $W$-$m$-$C$ for some $m> 1$ and
one of the following conditions holds.
\begin{itemize}
\item[\rm (i)] $p+q>m-1$, $q\leq m-1$ and $q< m-1 + \frac{\min(\theta,m)+\alpha-N}{N}p$;
\item[\rm (ii)] $p+q>m-1$, $q< m-1$ and $q= m-1 + \frac{\min(\theta,m)+\alpha-N}{N}p$;
\item[\rm (iii)] $p+q=m-1$ and $\min(\theta,m) > -\alpha$;
\item[\rm (iv)] $p+q<m-1$ and $\min(\theta,m) \geq -\alpha$.
\end{itemize}
Then  \eqref{gog} has no positive solutions.
\end{theorem}

\begin{proof}
(i) Assume first that $q= m-1$ which also implies $N< \min(\theta,m)+\alpha$. By Lemma \ref{l1},
\begin{equation*}
\int_{B_{2R}\setminus B_1} u^p dx\leq CR^{N-\min(\theta,m)-\alpha} \quad\mbox{ for all }R> 2.
\end{equation*}
Letting $R\to \infty$ in the above estimate, we conclude that \eqref{go} has no positive solutions.

Assume next that $q< m-1$.  Using H\"older's inequality, we estimate
\begin{equation}\label{pr1}
\int_{B_{2R}\setminus B_R} 1 \leq \Big(\int_{B_{2R}\setminus B_R}u^p\Big)^{\frac{m-1-q}{p+m-1-q}}
\Big(\int_{B_{2R}\setminus B_R}u^{q-m+1}\Big)^{\frac{p}{p+m-1-q}},
\end{equation}
which we rewrite as
\begin{equation*}
CR^N\leq \Big[\Big(\int_{B_{2R}\setminus B_R}u^p\Big)\Big(\int_{B_{2R}\setminus B_R}u^{q-m+1}\Big)
\Big]^{\frac{m-1-q}{p+m-1-q}}\Big(\int_{B_{2R}\setminus B_R}u^{q-m+1}\Big)^{\frac{p-m+1+q}{p+m-1-q}}.
\end{equation*}
Now, by Lemma \ref{l1} we deduce that
\begin{equation*}
CR^N\leq C\big(R^{2N-\min(\theta,m)-\alpha}\big)^{\frac{m-1-q}{p+m-1-q}}
\Big(\int_{B_{2R}\setminus B_R}u^{q-m+1}\Big)^{\frac{p-m+1+q}{p+m-1-q}}
\end{equation*}
for all $R>2$ which yields
\begin{equation}\label{pr2}
\int_{B_{2R}\setminus B_R}u^{q-m+1}
\geq CR^{N+\frac{(\min(\theta,m)+\alpha)(m-1-q)}{p+q-m+1}} \quad\mbox{ for all }R> 2.
\end{equation}
Again by Lemma \ref{l1} we have
\begin{equation}\label{pr3}
\Big(\int_{B_{2R}\setminus B_1}u^{p}\Big)\Big(\int_{B_{2R}\setminus B_R}u^{q-m+1}\Big)
\leq CR^{2N-\min(\theta,m)-\alpha} \quad\mbox{ for all }R> 2.
\end{equation}
Therefore,
\begin{equation}\label{pr4}
\int_{B_{2R}\setminus B_R} u^{q-m+1}
\leq \frac{CR^{2N-\min(\theta,m)-\alpha}}{\int_{B_{4}\setminus B_1}u^p}
\leq CR^{2N-\min(\theta,m)-\alpha}.
\end{equation}
From \eqref{pr2} and \eqref{pr4} we deduce
\begin{equation*}
C_1 R^{N+\frac{(\min(\theta,m)+\alpha)(m-1-q)}{p+q-m+1}}\leq \int_{B_{2R}\setminus B_R}u^{q-m+1}
\leq C_2 R^{2N-\min(\theta,m)-\alpha} \quad\mbox{ for all }R> 2.
\end{equation*}
Since $q< m-1 +\frac{\min(\theta,m)+\alpha-N}{N}p$, the above inequality cannot hold for large
enough $R> 2$. Hence, equation \eqref{gog} cannot have positive solutions.

(ii) When $q<m-1$ and $q= m-1 + \frac{\min(\theta,m)+\alpha-N}{N}p$, the argument above gives
\begin{equation}\label{ee0}
C_1 R^{2N-\min(\theta,m)-\alpha} \leq \int_{B_{2R} \backslash B_R} u^{q-m+1}
\leq C_2 R^{2N-\min(\theta,m)-\alpha} \quad\mbox{ for all }R> 2.
\end{equation}
Since $$\frac{p}{p+m-1-q} = \frac{N}{2N-\min(\theta,m)-\alpha},$$ estimate \eqref{pr1} yields
\begin{equation*}
CR^N \leq \left( \int_{B_{2R} \backslash B_R} u^p \right)^{\frac{m-1-q}{p+m-1-q}}
\cdot CR^N \quad\mbox{ for all }R> 2.
\end{equation*}
Hence
\begin{equation*}
\int_{B_{2R} \backslash B_R} u^p \geq C \quad\mbox{ for all }R> 2,
\end{equation*}
which shows that $\int_{\R^N\setminus B_1} u^p = \infty$ and then $\int_{B_{2R}\setminus B_1} u^p
\rightarrow \infty$ as $R\rightarrow \infty$. Further, from \eqref{pr3} we have
\begin{equation*}
\int_{B_{2R}\setminus B_R}u^{q-m+1}
\leq \frac{CR^{2N-\min(\theta,m)-\alpha}}{\int_{B_{2R}\setminus B_1}u^p}=
o(R^{2N-\min(\theta,m)-\alpha}) \quad\mbox{ as } R\rightarrow \infty,
\end{equation*}
which contradicts the first estimate in \eqref{ee0}.

(iii) Assume that $p+q= m-1$ and $\min(\theta,m)>-\alpha$. By Lemma \ref{l1}, we have
\begin{equation}\label{pq1}
\Big( \int_{B_{2R} \setminus B_1} u^p \Big) \Big( \int_{B_{2R}\setminus B_R} u^{-p} \Big)
\leq CR^{2N-\min(\theta,m)-\alpha} \quad\mbox{ for all }R> 2.
\end{equation}
On the other hand, H\"older's inequality gives
\begin{equation}\label{pq1.2}
CR^{2N}= \Big(\int_{B_{2R}\setminus B_R} 1 \Big)^2\leq \Big(\int_{B_{2R}\setminus B_R}u^p\Big)
\Big(\int_{B_{2R}\setminus B_R}u^{-p}\Big) \quad\mbox{ for all }R> 2.
\end{equation}
Since $\min(\theta,m)+\alpha>0$, the estimates \eqref{pq1} and \eqref{pq1.2} cannot hold for large
enough $R>2$. This shows that \eqref{gog} cannot have positive solutions.

(iv) Assume that $p+q< m-1$ and $\min(\theta,m)\geq -\alpha$. We apply H\"older's inequality to
derive \eqref{pr1} which we may rewrite as
\begin{equation*}
CR^N\leq \Big(\int_{B_{2R}\setminus B_R}u^p\Big)^{\frac{m-1-p-q}{p+m-1-q}}
\Big[\Big(\int_{B_{2R}\setminus B_R}u^{p}\Big)
\Big(\int_{B_{2R}\setminus B_R}u^{q-m+1}\Big)\Big]^{\frac{p}{p+m-1-q}}.
\end{equation*}
Using the estimate in Lemma \ref{l1}, we find
\begin{equation*}
R^N \leq C\Big(R^{2N-\min(\theta,m)-\alpha}\Big)^{\frac{p}{p+m-1-q}}
\Big(\int_{B_{2R}\setminus B_R}u^p\Big)^{\frac{m-1-p-q}{p+m-1-q}} \quad\mbox{ for all }R> 2,
\end{equation*}
which implies
\begin{equation}\label{pr5}
\int_{B_{2R}\setminus B_R}u^p \geq
CR^{N+\frac{p(\min(\theta,m)+\alpha)}{m-1-p-q}} \quad\mbox{ for all }R> 2.
\end{equation}
On the other hand, since $I_\alpha\ast u^p<\infty$ yields
\begin{equation*}
\int_{\R^N\setminus B_1}\frac{u(x)^p}{|x|^{N-\alpha}} \,dx< \infty,
\end{equation*}
which further implies
\begin{equation*}
R^{\alpha-N}\int_{B_{2R}\setminus B_R}u^p \leq C \quad\mbox{ for all }R> 2.
\end{equation*}
Combining this with \eqref{pr5}, we find
\begin{equation*}
C_1 R^{N-\alpha}\geq \int_{B_{2R}\setminus B_R}u^p
\geq C_2 R^{N+\frac{p(\min(\theta,m)+\alpha)}{m-1-p-q}} \quad\mbox{ for all }R> 2,
\end{equation*}
which gives a contradiction because $N-\alpha< N \leq N+\frac{p(\min(\theta,m)+\alpha)}{m-1-p-q}$
by assumption.
\end{proof}

\section{Proof of the main results}

We first provide two nonexistence results related to \eqref{go}.

\begin{proposition}\label{mL1}
Let $p>0$, $q,\mu\in \R$ and $\alpha\in (0, N)$.

If one of the following conditions hold:
\begin{enumerate}
\item[\rm (i)] $m-1< q\leq m-1-\frac{N-m-\alpha}{|\beta^{-}|}$ and $\, \alpha> N-m$;
\item[\rm (ii)] $m-1\leq p+q\leq m-1+\frac{m+\alpha}{|\beta^{-}|}$;
\end{enumerate}
then \eqref{go} has no positive solutions.

\end{proposition}
\begin{proof}
(i) Using Lemma \ref{lbas}(i), we have
$(I_\alpha*u^p)(x)\geq c|x|^{\alpha-N}$ in $\R^N\setminus B_2$.
Therefore, inequality \eqref{go} implies
\begin{equation}\label{go1}
-\Delta_m u - \frac{\mu}{|x|^m}u^{m-1} \geq c|x|^{\alpha-N}u^q \quad\mbox{ in }\R^N\setminus B_2.
\end{equation}
Taking $\sigma= N-\alpha$ in Proposition \ref{p2}(i), we find that if
$$
N-\alpha\leq \beta^{-}(q-m+1)+m,
$$
which further yields
$$
q\leq m-1-\frac{N-m-\alpha}{|\beta^{-}|},
$$
then, inequality \eqref{go} has no positive solution.

(ii)  We note first that $\beta^-<0$ otherwise, from \eqref{betam} and Lemma \ref{lbas}(ii) we deduce $I_\alpha\ast u^p=\infty$, contradiction.  Hence, $\beta^-<0$ wich corresponds to $-\infty <\mu\leq C_H$ if $N>m$, and $-\infty<\mu<0$ if $N\leq m$.

We divide our argument into two cases.

\noindent{ \bf Case 1:}  $p+q< m-1+\frac{m+\alpha}{|\beta^{-}|}$. By H\"older inequality we estimate
\begin{equation*}
\Big(\int_{B_{2R}\setminus B_R}u^p\Big)\Big(\int_{B_{2R}\setminus B_R}u^{q-m+1}\Big)\geq
\Big(\int_{B_{2R}\setminus B_R}u^{\frac{p+q-m+1}{2}}\Big)^2 \quad\mbox{ for all } R> 2.
\end{equation*}
Using Lemma \ref{l1} (taking $\theta= m$) together with the estimate $u\geq C|x|^{\beta^{-}}$ in $\R^N\setminus
B_2$ we find
\begin{equation*}
C_1 R^{2N-m-\alpha}\geq \Big(\int_{B_{2R}\setminus B_R}u^{\frac{p+q-m+1}{2}}\Big)^2
\geq C_2 R^{2N+(p+q-m+1)\beta^{-}} \quad\mbox{ for all }R> 2.
\end{equation*}
However, this is a contradiction as $2N-m-\alpha< 2N+(p+q-m+1)\beta^{-} $.

\noindent{ \bf Case 2:}  $p+q= m-1+\frac{m+\alpha}{|\beta^{-}|}$.

In view of \eqref{betam} and Lemma \ref{lbas}(ii) we must have
$$
p> \frac{\alpha}{|\beta^{-}|}\quad\mbox{ and }\quad I_\alpha\ast u^p\geq c|x|^{\alpha+p\beta^-}\; \mbox{ in }\; \R^N\setminus B_2.
$$
Thus, $u$ satisfies
\begin{equation}\label{eqA1}
-\Delta_m u - \frac{\mu}{|x|^m}u^{m-1} \geq c|x|^{\alpha+p\beta^-} u^{q} \quad\mbox{ in } \R^N\setminus B_2.
\end{equation}

Assume first $q\geq m-1$. Then, by Proposition \ref{p2}(i)  with $\sigma=-\alpha+p|\beta^-|$ we deduce that if $p+q= m-1+\frac{m+\alpha}{|\beta^{-}|}$ then, inequality \eqref{go} has no positive solutions.

We discuss next the case  $q<m-1$. By estimate \eqref{pr1} (which follows from H\"older's inequality) and Lemma \ref{l1} we find
$$
\begin{aligned}
CR^N&=\int_{B_{2R}\setminus B_R} 1
\leq \Big(\int_{B_{2R}\setminus B_R}u^p\Big)^{\frac{m-1-q}{p+m-1-q}} \Big(\int_{B_{2R}\setminus B_R}u^{q-m+1}\Big)^{\frac{p}{p+m-1-q}}\\
&=\Big(\int_{B_{2R}\setminus B_R}u^p\Big)^{\frac{m-1-q-p}{p+m-1-q}}
\Big[\Big(\int_{B_{2R}\setminus B_R}u^p\Big)\Big(\int_{B_{2R}\setminus B_R}u^{q-m+1}\Big)
\Big]^{\frac{p}{p+m-1-q}}\\
&\leq c \big(R^{2N-m-\alpha}\big)^{\frac{p}{p+m-1-q}} \Big(\int_{B_{2R}\setminus B_R}u^p\Big)^{\frac{m-1-q-p}{p+m-1-q}} .
\end{aligned}
$$
Since $p+q>m-1$ the above inequality yields
\begin{equation}\label{etm1}
\int_{B_{2R}\setminus B_R}u^p\leq C R^{N+p\beta^-} \quad\mbox{ for all }R>2.
\end{equation}
Let us observe now that for $\kappa>0$ sufficiently small and $0<\tau<1/(m-1-q)$, the function
$$
v(x)=\kappa |x|^{\beta^-}\log^\tau(s|x|)
$$
satisfies
\begin{equation}\label{eqAA7}
-\Delta_m v - \frac{\mu}{|x|^m}v^{m-1} \leq c|x|^{\alpha+p\beta^-} v^{q} \quad\mbox{ in } \R^N\setminus B_2,
\end{equation}
where $c>0$ is the constant from \eqref{eqA1}.
Indeed, the left-hand side in \eqref{eqAA7} can be computed using \eqref{radial2} in which  we have $A=G(\beta^-)-\mu=0$ in \eqref{abc2}. Thus, the power of the logarithmic term in \eqref{radial2} is at most $\tau(m-1)-1<\tau q$ which shows that $v$ fulfills \eqref{eqAA7}. Taking now $\kappa>0$ small enough such that $u\geq v$ on $\partial B_2$, by the comparison principle in \cite[Theorem 2.1]{LLM2007} it follows that $u\geq v$ in $\R^N\setminus B_2$. Thus,
$$
u(x)\geq  \kappa |x|^{\beta^-}\log^\tau(s|x|)\quad\mbox{ in }\R^N\setminus B_2.
$$
Now, from the above estimate and \eqref{etm1} we find
$$
cR^{N+p\beta^-}\log^{p\tau} (sR)\leq   \int_{B_{2R}\setminus B_R}u^p\leq C R^{N+p\beta^-} \quad\mbox{ for all }R>2,
$$
which is a contradiction for large $R>2$, since $\tau>0$.
\end{proof}

In the next result we provide conditions for nonexistence of a positive solution to \eqref{go} separately for $m\geq N$ and $N>m$.

\begin{proposition}\label{mL2}
Let $p>0$, $q,\mu\in \R$.
\begin{enumerate}
\item[\rm (i)] If $m\geq N$ and one of the following conditions hold:
\begin{enumerate}
\item[\rm (i1)] $0\leq \mu \leq C_H$;
\item[\rm (i2)] $\mu< 0$ and $p \leq \frac{\alpha}{|\beta^{-}|}$;
\end{enumerate}
then \eqref{go} has no positive solutions.
\item[\rm (ii)] If $N>m> 1$ and one of the following conditions hold:
\begin{enumerate}
\item[\rm (ii1)] $p\leq \frac{\alpha}{|\beta^{-}|} $;
\item[{\rm (ii2)}] $q\leq m-1-\frac{N-m-\alpha}{|\beta^+|}$, $\mu>0$ and $\alpha<N-m$;
\end{enumerate}
then \eqref{go} has no positive solutions.
\end{enumerate}
\end{proposition}

\begin{proof} We proceed by contradiction and suppose that \eqref{go} has a positive solution $u$. By Proposition \ref{p1} we have $u\geq c|x|^{\beta^-}$ in $\R^N\setminus \overline B_2$.

(i1) Assume $m\geq N$ and $0\leq \mu \leq C_H$. Then $\beta^{-}\geq 0$ by Lemma \ref{mu}(iii) and
so $\alpha+p\beta^{-}>0$. Therefore, by Lemma \ref{lbas}(ii), $(I_\alpha*u^p)(x)=\infty$ for all
$x\in \R^N\setminus B_2$, which contradicts the fact that $(I_\alpha*u^p)u^q \in
L^{1}_{loc}(\R^N\setminus \overline B_1)$. Hence, \eqref{go} has no positive solution.

(i2) Assume $m\geq N$ and $\mu< 0$. Then $\beta^- < 0$ by Lemma \ref{mu}(iii) and so
\begin{equation*}
\alpha + p\beta^- = \alpha - p|\beta^-| \geq 0
\end{equation*}
by condition (i2).  Hence, by Lemma \ref{lbas}(ii), $(I_\alpha*u^p)(x)=\infty$ for all $x\in
\R^N\setminus B_2$, which contradicts the fact that $(I_\alpha*u^p)u^q \in L^{1}_{loc}( \R^N\setminus \overline B_1)$ and
\eqref{go} has no positive solution.

(ii1) When $N>m$, we have $\beta^- < 0$ by Lemma \ref{mu}(iii), so one may proceed as in (i2) to
conclude that \eqref{go} has no positive solution.

(ii2) By Lemma \ref{lbas}(i) we have that $u$ is a solution of \eqref{go1} for some constant $c>0$.
Since $N>m$ and $\mu>0$ we have $\beta^+<0$.

If $\mu<C_H$ we apply Proposition \ref{p2}(ii) with $\sigma=N-\alpha$ to deduce that \eqref{go} has no solutions. If $\mu=C_H$ then Proposition \ref{p2}(iii)-(iv) with $\sigma=N-\alpha$ yields the same conclusion.

\end{proof}

\subsection{Proof of Theorem \ref{thmain1}}

Conditions (i)-(iii) follow from Proposition \ref{p1}, Proposition \ref{mL2}(ii1), Proposition
\ref{mL1}(ii) and Theorem \ref{thm2}(iii)-(iv). Also, condition (iv) in Theorem \ref{thmain1}
follows from Proposition \ref{mL1}(i), Proposition \ref{mL2}(ii2) and Theorem \ref{thm2}(ii)-(iv).

Suppose now that conditions (i)-(iv) in Theorem \ref{thmain1} hold.  To construct a positive
solution of \eqref{go}, we divide our argument into two cases.

\noindent{\bf Case 1:} $\mu<C_H$.  In this case, we construct a solution of the form
\begin{equation}\label{u1}
u(x)= \kappa |x|^\gamma,
\end{equation}
where $\beta^- < \gamma < \min\{ \beta^+ , 0\}$ and $\kappa>0$ is sufficiently small.  According to
\eqref{radial1}, one has
\begin{equation*}
-\Delta_m u -\frac{\mu}{|x|^m}u^{m-1}= \kappa^{m-1} (G(\gamma)-\mu)|x|^{\gamma(m-1)-m} \quad\mbox{ in }\R^N\setminus B_1.
\end{equation*}
Since $\beta^- < \gamma < \beta^+$ by assumption, the constant $G(\gamma)-\mu$ is positive and we
get
\begin{equation}\label{exs1}
-\Delta_m u -\frac{\mu}{|x|^m}u^{m-1}= C_1\kappa^{m-1} |x|^{\gamma(m-1)-m} \quad\mbox{ in }\R^N\setminus B_1.
\end{equation}

\noindent{\bf Subcase 1a:} If $-\frac Np \leq \beta^-$, then we may choose the exponent $\gamma$ so
that
\begin{equation} \label{b1}
-\frac Np \leq \beta^- < \gamma < \min \left\{ \beta^+ , -\frac{\alpha}{p}, -\frac{m+\alpha}{p+q-m+1} \right\}.
\end{equation}
This is possible because of conditions (ii) and (iii) in Theorem \ref{thmain1}.  Since
$p|\gamma|<N$, it then follows by Lemma \ref{lbas}(iii) with $\tau=0$ that
\begin{equation*}
(I_\alpha*u^p) \cdot u^q \leq C\kappa^p |x|^{\alpha+p\gamma} \cdot u^q
= C\kappa^{p+q} |x|^{\alpha+(p+q)\gamma} \quad\mbox{ in }\R^N\setminus B_1.
\end{equation*}
Using the upper bound in \eqref{b1} along with condition (iii), we conclude that
\begin{align*}
(I_\alpha*u^p) \cdot u^q &\leq C\kappa^{p+q} |x|^{\gamma(m-1)-m}
\leq C_1\kappa^{m-1} |x|^{\gamma(m-1)-m} \quad\mbox{ in }\R^N\setminus B_1
\end{align*}
for all small enough $\kappa>0$.  In view of \eqref{exs1}, this implies that $u$ is a solution of
\eqref{go}.

\noindent{\bf Subcase 1b:} If $\beta^- < -\frac Np$ and $\alpha>N-m$, then $q>m-1$ by condition
(iv)$_1$ and one may use this condition to find some
\begin{equation} \label{b2}
\beta^- < \gamma < \min \left\{ \beta^+ , -\frac{N}{p}, \frac{N-m-\alpha}{q-m+1} \right\}.
\end{equation}
Since $p|\gamma|>N$, it then follows by Lemma \ref{lbas}(iii) with $\tau=0$ that
\begin{equation} \label{exs2}
(I_\alpha*u^p) \cdot u^q \leq C\kappa^p |x|^{\alpha-N} \cdot u^q
= C\kappa^{p+q} |x|^{\alpha-N+q\gamma} \quad\mbox{ in }\R^N\setminus B_1.
\end{equation}
Using the upper bound in \eqref{b2} along with condition (iii), we conclude that
\begin{align*}
(I_\alpha*u^p) \cdot u^q &\leq C\kappa^{p+q} |x|^{\gamma(m-1)-m}
\leq C_1\kappa^{m-1} |x|^{\gamma(m-1)-m} \quad\mbox{ in }\R^N\setminus B_1
\end{align*}
for all small enough $\kappa>0$.  Once again, this implies that $u$ is a solution of \eqref{go}.

\noindent{\bf Subcase 1c:} If $\beta^- < -\frac Np$ and $\alpha\leq N-m$ and $q\geq m-1$, we choose
$\gamma$ so that
\begin{equation} \label{b3}
\beta^- < \gamma < \min \left\{ \beta^+ , -\frac{N}{p} \right\}.
\end{equation}
Since $p|\gamma|>N$, our estimate \eqref{exs2} remains valid and we still have
\begin{align*}
(I_\alpha*u^p) \cdot u^q \leq C\kappa^{p+q} |x|^{\alpha-N+q\gamma}
\leq C_1\kappa^{m-1} |x|^{-m + \gamma(m-1)} \quad\mbox{ in }\R^N\setminus B_1
\end{align*}
for all small enough $\kappa>0$ because $\alpha-N\leq -m$ and $q\gamma \leq (m-1)\gamma$.

\noindent{\bf Subcase 1d:} If $\beta^- < -\frac Np$ and $\alpha< N-m$ and $q<m-1$, we choose
$\gamma$ so that
\begin{equation} \label{b4}
\max \left\{ \beta^- , -\frac{N-m-\alpha}{m-1-q} \right\} <
\gamma < \min \left\{ \beta^+ , -\frac{N}{p} \right\}.
\end{equation}
This is possible because of conditions (iv)$_3$ and (iv)$_4$ in Theorem \ref{thmain1}.  Condition
(iv)$_4$ is only needed when $\mu>0$, because $\beta^+ \geq 0$ when $\mu\leq 0$. Since
$p|\gamma|>N$, we get
\begin{align*}
(I_\alpha*u^p) \cdot u^q \leq C\kappa^{p+q} |x|^{\alpha-N+q\gamma}
\leq C\kappa^{p+q} |x|^{\gamma(m-1)-m} \quad\mbox{ in }\R^N\setminus B_1
\end{align*}
because of the lower bound in \eqref{b4}.  Thus, the result follows exactly as before.

\noindent{\bf Case 2:} $\mu=C_H$. In this case, we have
\begin{equation*}
\beta^- = \beta^+ = \beta_* = \frac{m-N}{m}
\end{equation*}
and we look for a solution of the form
\begin{equation}\label{u2}
u(x)= \kappa |x|^{\beta_*} \log^\tau \big(s|x|\big),
\end{equation}
where $\tau\in (0, 2/m)$ and $s>1$ are chosen so that $|\beta_*|\log s>\tau$.
We have that $u$ satisfies \eqref{radial2} in which $a,b,c$ defined in \eqref{abc1} are given by
\begin{equation*}
a=\beta_*^2\,, \quad b=-\tau \beta_*(m-2)\,,\quad c=-\tau(\tau-1)(m-1).
\end{equation*}
Thus, the coefficients $A$, $B$ and $C$ defined in \eqref{abc2} are now given by
\begin{equation*}
A=B=0\,,\quad C=\frac{1}{2}\tau |\beta_*|^{m-2}(m-1)(2-m\tau)>0.
\end{equation*}
Hence, by taking $s>1$ large enough we may ensure from \eqref{radial2} that
\begin{equation}\label{exs3}
-\Delta_m u -\frac{\mu}{|x|^m}u^{m-1}\geq c \kappa^{m-1} |x|^{(m-1)\beta_*-m}
\log^{\tau(m-1)-2}\big(s |x|\big) \quad\mbox{ in }\R^N\setminus B_1.
\end{equation}
With $u$ given by \eqref{u2}, from Lemma \ref{lbas}(iii) we have
\begin{equation*}
(I_\alpha*u^p)u^q(x)\leq
C\kappa^{p+q}\left\{
\begin{aligned}
&|x|^{\alpha+(p+q)\beta_*}\log^{(p+q)\tau}(s|x|)&&\quad\mbox{ if }p|\beta_*|<N\\
&|x|^{\alpha-N+q\beta_*}\log^{1+(p+q)\tau}(s|x|)&&\quad\mbox{ if }p|\beta_*|=N\\
&|x|^{\alpha-N+q\beta_*}\log^{q\tau}(s|x|)&&\quad\mbox{ if }p|\beta_*|>N
\end{aligned}
\right. \quad\mbox{ in }\; \R^N\setminus B_1.
\end{equation*}
In order to conclude the proof, we only need to check that
\begin{equation}\label{eqqab}
\left\{
\begin{aligned}
&(m-1)\beta_*-m>\alpha+(p+q)\beta_*&&\quad\mbox{ if } p|\beta_*|<N,\\
&(m-1)\beta_*-m>\alpha-N+q\beta_*&&\quad\mbox{ if } p|\beta_*|\geq N.
\end{aligned}
\right.
\end{equation}
The first condition in \eqref{eqqab} follows directly from condition (iii) in Theorem
\ref{thmain1}. The second condition in \eqref{eqqab} follows from conditions (iv)$_1$, (iv)$_2$ and
(iv)$_4$ in Theorem \ref{thmain1}; note that in this setting, the case $q=m-1$ and $p|\beta_*|=N$
cannot occur. \qed

\subsection{Proof of Theorem \ref{thmain2}}

Conditions (i)-(iv) follow from Proposition \ref{p1}, Proposition \ref{mL2}(i1)-(i2), Proposition \ref{mL1}(i)-(ii) and Theorem \ref{thm2}(iii)-(iv).

Conversely, if conditions (i)-(iv) in Theorem \ref{thmain2} hold, then one may construct a positive
solution of \eqref{go} by using the same method as in Case 1 of the proof of Theorem \ref{thmain1}.
More precisely, one can show that \eqref{go} has a solution of the form \eqref{u1}.  Since $m\geq
N$ and $\mu<0$, one has $\beta^- < 0 < \beta^+$ by Lemma \ref{mu}(iii).  In particular, our
previous approach applies almost verbatim and one may choose the exponent $\gamma$ exactly as
before for each of the four subcases.  Since the proof is almost identical, we shall omit the
details.

\qed

\noindent{\bf Acknowledgments.}  The authors would like to thank the anonymous referee for the careful reading of our manuscript and for suggestions which lead to the current work.
The third named author, Gurpreet Singh, acknowledges the financial support of The Irish Research
Council Postdoctoral Scholarship, number R13165.

\end{document}